\documentclass[review]{elsarticle}
\usepackage{stmaryrd}
\usepackage{bbm}
\usepackage[dvipsnames]{xcolor}
\usepackage{amsmath}
\usepackage{amsthm}
\usepackage{amssymb}
\usepackage{mathtools}
\usepackage{subcaption}
\usepackage{etoolbox}
\usepackage{environ}
\usepackage{ulem}
\usepackage{multicol}
\usepackage{listings}
\usepackage{graphicx}
\usepackage{multirow}
\usepackage{subcaption}
\usepackage{blkarray}
\usepackage{algorithm}
\usepackage{algpseudocode}

\usepackage{lineno,hyperref}
\modulolinenumbers[5]

\usepackage{stmaryrd}

\newcommand{\suchthat}{\;\ifnum\currentgrouptype=16 \middle\fi|\;}

\newcommand{\argmax}{\operatornamewithlimits{arg\,max}}
\newcommand{\argmin}{\operatornamewithlimits{arg\,min}}

\newcommand{\MID}{\operatorname{mid}}
\newcommand{\bc}{\operatorname{bc}}
\newcommand{\level}{\operatorname{level}}

\newcommand{\bp}{\operatorname{bp}}

\newcommand{\ord}{\operatorname{ord}}
\newcommand{\mc}{\operatorname{mc}}

\newcommand{\flag}{\operatorname{flag}}
\newcommand{\True}{\operatorname{True}}
\newcommand{\False}{\operatorname{False}}
\newcommand{\PP}{\operatorname{P}}
\newcommand{\NP}{\operatorname{NP}}

\newcommand{\lab}{\operatorname{label}}

\newcommand{\prevc}{\operatorname{prev\_card}}
\newcommand{\newc}{\operatorname{new\_card}}
\newcommand{\clique}{\operatorname{clique}}

\newtheorem*{rep@theorem}{\rep@title}
\newcommand{\newreptheorem}[2]{%
\newenvironment{rep#1}[1]{%
 \def\rep@title{#2 \ref{##1}}%
 \begin{rep@theorem}}%
 {\end{rep@theorem}}}
\makeatother

\newtheorem{corollary}{Corollary}
\newtheorem{theorem}{Theorem}
\newreptheorem{theorem}{Theorem}
\newtheorem{lemma}{Lemma}
\newreptheorem{lemma}{Lemma}
\newtheorem{remark}{Remark}
\newtheorem{definition}{Definition}

\newreptheorem{claim}{Claim}
\newtheorem{proposition}{Proposition}
\newreptheorem{proposition}{Proposition}

\usepackage{tikz,tkz-graph}
\tikzstyle{vertex}=[circle, draw, inner sep=0pt, minimum size=1.5em]
\tikzstyle{svertex}=[draw, inner sep=0pt, minimum size=1.5em]

\usetikzlibrary{decorations}
\usetikzlibrary{arrows.meta,patterns}

\journal{ }

\begin{document}

\begin{frontmatter}

\title{Maximal Clique and Edge-Ranking Bounds of Biclique Cover Number}

\author[1]{Bochuan Lyu\corref{cor1}}
\ead{bl46@rice.edu}
\author[1]{Illya V. Hicks}
\ead{ivhicks@rice.edu}
\cortext[cor1]{Corresponding author}
\address[1]{Rice University, Department of Computational Applied Mathematics and Operations Research, United States of America}

\begin{abstract}
The biclique cover number $(\bc)$ of a graph $G$ denotes the minimum number of complete bipartite (biclique) subgraphs to cover all the edges of the graph. In this paper, we show that $\bc(G) \geq \lceil \log_2(\mc(G^c)) \rceil \geq \lceil \log_2(\chi(G)) \rceil$ for an arbitrary graph $G$, where $\chi(G)$ is the chromatic number of $G$ and $\mc(G^c)$ is the number of maximal cliques of the complementary graph $G^c$, i.e., the number of maximal independent sets of $G$. We also show that $\lceil \log_2(\mc(G^c)) \rceil$ could be a strictly tighter lower bound of the biclique cover number than other existing lower bounds. We can also provide a bound of $\bc(G)$ with respect to the biclique partition number ($\bp$) of $G$: $\bc(G) \geq \lceil \log_2(\bp(G) + 1) \rceil$ or $\bp(G) \leq 2^{\bc(G)} - 1$ if $G$ is co-chordal. Furthermore, we show that $\bc(G) \leq \chi_r'(\mathcal{T}_{\mathcal{K}^c})$, where $G$ is a co-chordal graph such that each vertex is in at most two maximal independent sets and $\chi_r'(\mathcal{T}_{\mathcal{K}^c})$ is the optimal edge-ranking number of a clique tree of $G^c$.
\end{abstract}


\begin{keyword}
Biclique Covers, Maximal Cliques, Biclique Partitions, Edge-Rankings of Trees, Co-Chordal Graphs
\end{keyword}

\end{frontmatter}

\nolinenumbers

\section{Introduction}
The biclique cover number $(\bc)$ of a graph $G$ denotes the minimum number of complete bipartite (biclique) subgraphs to cover all the edges of the graph. The minimum biclique (edge) cover problem (MBCP) on a graph $G$ is referred to finding such a collection of bicliques, a biclique cover, with a cardinality of the biclique cover number of $G$. It was first proven to be NP-complete by Orlin~\cite{orlin1977contentment}. MBCP is not only complex on the general graph but also intricate even on some special classes of graphs. M{\"u}ller~\cite{muller1996edge} proved that MBCP remains NP-complete on chordal bipartitie graphs. Several research works focused on approximation algorithms for MBCP~\cite{chalermsook2014nearly, gruber2007inapproximability, simon1990approximate}, but it turned out that there doesn't exist an efficient algorithm even for approximating MBCP. It cannot be approximated in polynomial time to less than $O(|V|^{1 - \epsilon})$ or $O(|E|^{1/2 - \epsilon})$ factor for any $\epsilon > 0$ unless $\PP = \NP$ shown by~\cite{chalermsook2014nearly}. 

Despite of the hardness of solving MBCP on general graphs, the biclique cover number of some well-structured graphs are known, such as complete graphs and $2n$-vertex crown graphs~\cite{de1981boolean}, and grid graphs~\cite{guo2018biclique} and there exists polynomial time algorithm to solve the problem exactly on some special classes of graphs such as C4-free graphs~\cite{muller1996edge}, and domino-free graphs~\cite{amilhastre1998complexity}. 

One related problem to MBCP is boolean (binary) matrix factorization ($k$-BMF). The goal of $k$-BMF is to factorize 0-1 matrix with size $m \times n$ into the boolean product of two 0-1 matrices with sizes $m \times k$ and $k \times n$ under some norm. The boolean rank of a 0-1 matrix is the smallest number of $k$ for the exact factorization (the boolean product of the two matrices is exactly the orignal matrix). Monson et al.~\cite{monson1995survey} showed that finding a boolean rank of a 0-1 matrix is equivalent to solve MBCP on a bipartite graph represented by the matrix.


In this paper, we also build a connection between MBCP on co-chordal graph and another combinatorial optimization problem: optimal edge-rankings of trees. Iyer et al.~\cite{iyer1991edge} presented a polynomial-time approximation algorithm for optimal edge-rankings in trees with the worst case performance of ratio 2 and left whether the problem is NP-hard as an open question. Later, Torre et al.~\cite{de1995optimal} showed that the optimal solution can be found in polynomial time. Zhou and Nishizeki~\cite{zhou1994efficient,zhou1995finding} provided a $O(n^2)$-time algorithm for optimal edge-rankings in trees and then improved the time complexity to $O(n \log n)$. In a more recent work, Lam and Yue~\cite{lam2001optimal} designed a linear-time algorithm to solve the problem optimally.

In terms of applications, G{\"u}nl{\"u}k~\cite{gunluk2007new} used the biclique cover number to provide an upper bound for min-cut max-flow ratio for multicommodity flow problems. Epasto and Upfal~\cite{epasto2018efficient} studied two computational biology problems: human leukocyte antigen (HLA) serology~\cite{nau1978mathematical}, and Mod-Resc Parsimony Inference problem introduced by~\cite{nor2012mod}. In computer systems, Ene et al.~\cite{ene2008fast} studied heuristic of MBCP on bipartite graphs to solve role mining problem in role-based access control system. Another heuristic of MBCP on bipartite graphs is also studied by \cite{amilhastre1999fa} for a minimization of finite automata problem arised in constraint satisfaction problems. Biclique edge cover graph is also studied by~\cite{hirsch2006biclique} for confluent drawings, which is a field studying how to draw nonplanar graphs on a plane. Huchette and Vielma~\cite{huchette2019combinatorial} and Lyu et al.~\cite{lyu2022modeling} solved MBCP in order to find small and strong mixed-integer programming (MIP) formulations of disjunctive constraints. 

In Section~\ref{sec:pre_bc}, we will introduce some basic notations in graph theory. In Section~\ref{sec:lower_bound}, we will prove that the biclique cover number of a graph $G$ is no less than $\lceil \log_2(\mc(G^c)) \rceil$, where $\mc(G^c)$ is the number of maximal cliques of the complementary graph $G^c$. We also discuss that the new lower bound is not worse than the bound in~\cite{harary1977biparticity} and show that the new lower bound can be better than other existing bounds~\cite{fishburn1996bipartite,jukna2009covering}. In section~\ref{sec:upper_bound}, we will also show that $\bc(G) \geq \lceil \log_2(\bp(G) + 1) \rceil$ or $\bp(G) \leq 2^{\bc(G)} - 1$ given a co-chordal graph $G$, where $\bp(G)$ is the biclique partition number of $G$. It is a better bound than $\bp(G) \leq \frac{1}{2}(3^{\bc(G)} - 1)$~\cite{pinto2013biclique} if $G$ is co-chordal and $\bc(G) > 1$. Then, we focus on a heuristic of finding biclique cover on co-chordal graph in Section~\ref{sec:upper_bound}. By bounding the size of the biclique cover returned by the heuristic in Algorithm~\ref{alg:biclique_tot_bc}, the biclique cover number of a co-chordal graph $G$ is at most $\mc(G^c)-1$ and if each vertex in $G$ is in at most two maximal indepedent sets of $G$, we can provide a tighter upper bound of $G$ by using the optimal edge-ranking number of a clique tree of $G^c$. Then, we show that $\bc(G) = \lceil \log_2(\mc(G^c)) \rceil$ if each vertex in $G$ is in at most two maximal indepedent sets of $G$ and $G^c$ has a clique tree with an optimal edge-ranking number of $\lceil \log_2(\mc(G^c)) \rceil$. In Section~\ref{sec:c_fw_bc}, we summarize our contributions and discuss the potential future work.


\section{Preliminaries} \label{sec:pre_bc}

A \textit{simple graph} is a pair $G = (V, E)$ where $V$ is a finite set of vertices and the edge set $E \subseteq \{uv: u, v\in V, u \neq v\}$. We use $V(G)$ and $E(G)$ to denote the vertex set and edge set of the graph $G$ respectively. Two vertices are \textit{adjacent} in $G$ if there is an edge between them. The \textit{neighborhood} of a vertex $v$ of a graph $G$, $N_G(v)$, is the set of all vertices that are adjacent with $v$.

A \textit{subgraph} $G' = (V', E')$ of $G$ is a graph where $V' \subseteq V$ and $E' \subseteq \{uv \in E: u, v \in V'\}$. Given $A \subseteq V$, the \textit{subgraph} of $G$ \textit{induced} by $A$ is denoted as $G(A) = (A, E_A)$, where $E_A = \{uv \in E: u, v \in A\}$. A \textit{clique} is a subset of vertices of an graph $G$ such that every two distinct vertices are adjacent in $G$. A \textit{maximal clique} of $G$ is a clique of $G$ such that it is not a proper subset of any clique of $G$. Note that a vertex set with only one vertex $K_1$ is also a clique. We denote the number of maximal cliques of $G$ as $\mc(G)$ and we use $\mathcal{K}_G$ or $\mathcal{K}$ to denote the set of all maximal cliques of $G$. A \textit{maximum clique} of $G$ is a clique of $G$ with the maximum number of vertices and we denote that number as the clique number of $G$, $\omega(G)$. An \textit{independent set} of a graph $G$ is a set of vertices that are not adjacent with each other in $G$. Similarly, a \textit{maximal independent set} is an independent set that is not a proper subset of any independent set and a \textit{maximum independent set} is an independent set with the maximum number of vertices. Note that an independent set can be empty or only have one vertex.

The \textit{chromatic number} of a graph $G$, $\chi(G)$, is the smallest number of colors needed to color the vertices of $G$ such that no two adjacent vertices share the same color. A \textit{matching} in $G$ is a set of edges without common vertices and a \textit{maximum matching} of $G$ is a matching with the maximum size, which is denoted as $M(G)$. Given a graph $G$, $G_E$ is a graph with the edge set of $G$ as the vertex set and two vertices is adjacent if the corresponding edges have distinct end-nodes and are not included in a cycle of length 4 in $G$.

Given two vertex sets $U$ and $V$, we denote $U \times V$ to be the edge set $\{uv: u \in U, v \in V\}$. A \textit{bipartite} graph $G = (L \cup R, E)$ is a graph where $L$ and $R$ are disjointed vertex sets with the edge set $E \subseteq L \times R$. A \textit{biclique} graph is a complete bipartite graph $G = (L \cup R, E)$ where $E = L \times R$ and we denote it as $\{L, R\}$ for short. A \textit{biclique cover} of a graph $G$ is a collection of biclique subgraphs of $G$ such that every edge of $G$ is in at least one biclique of the collection. The \textit{minimum biclique cover problem} (MBCP) on $G$ is to find a biclique cover with the minimum number of bicliques in the collection and we denote that value to be $\bc(G)$. A \textit{biclique partition} of a graph $G$ is a biclique cover of $G$ that cover each edge exactly once and the biclique partition number of $G$, $\bp(G)$, is the size of the minimum biclique partition of $G$.

A graph $C_n = (V, E)$ is a \textit{cycle} if the vertices and edges: $V = \{v_1, v_2, \hdots, v_n\}$ and $E = \{v_1v_2, v_2v_3, \hdots, v_{n-1}v_n, v_nv_1\}$. A graph is a \textit{tree} if it is connected and does not have any subgraph that is a cycle. We denote that $\llbracket n \rrbracket = \{1, 2, \hdots, n\}$ where $n$ is a positive integer. A vertex is \textit{simplicial} if its neighborhood is a clique. An ordering $v_1, v_2, \hdots, v_n$ of $V$ is a \textit{perfect elimination ordering} if for all $i \in \llbracket n \rrbracket$, $v_i$ is simplicial on the induced subgraph $G(\{v_j: j \in \{i, i+1, \hdots, n\}\})$. An ordering function $\sigma: \llbracket n \rrbracket \rightarrow V$ is defined to describe the ordering of vertices $V$. A \textit{clique tree} $\mathcal{T}_{\mathcal{K}}$ for a chordal graph $G$ is a tree where each vertex represents a maximal clique of $G$ and satisfies the \textit{clique-intersection property}: given any two distinct maximal cliques $K^1$ and $K^2$ in the tree, every clique on the path between $K^1$ and $K^2$ in $\mathcal{T}_{\mathcal{K}}$ contains $K^1 \cap K^2$. We also define the \textit{middle set} of edge $e \in \mathcal{T}_{\mathcal{K}}$, $\MID(e)$, to be the intersection of the vertices of cliques on its two ends.

A graph is a \textit{windmill} graph $D^{(m)}_n$ if it can be obtained by taking $m$ copies of the complete graph $K_n$ sharing with a universal vertex. We use the notation of \textit{co-} to represent the complementary graph, for example, a graph is \textit{co-windmill} if its complementary graph is windmill. 

A graph is \textit{chordal} if there is no induced cycle subgraph of length greater than 3. A graph is chordal if and only if it has a clique tree~\cite{blair1993introduction}. Also, a graph is chordal if and only if it has a perfect elimination ordering~\cite{fulkerson1965incidence}. The perfect elimination ordering of a chordal graph can be obtained by \textit{lexicographic breadth-first search} (LexBFS)~\cite{rose1976algorithmic}. Note that the LexBFS algorithm can be implemented in linear-time: $O(|V| + |E|)$ with partition refinement~\cite{habib2000lex}. Another approach to produce perfect elimination ordering of a chordal graph is maximum cardinality search (MCS) introduced by Tarjan in unpublished lecture notes. At each step. algorithm selects to label an unlabeled vertex adjacent to the largest number of labeled vertices as the next vertex (label arbitrary one if there are ties). Tarjan and Yannakaki~\cite{tarjan1984simple} provided a detailed MCS algorithm that can run in $O(|V| + |E|)$ time. Furthermore, Blair and Peyton~\cite{blair1993introduction} expanded the MCS algorithm to compute a clique tree of a chordal graph $G$ in polynomial time.

\section{A New Lower Bound of Biclique Cover Number} \label{sec:lower_bound}

In this section, we will derive a new lower bound of biclique cover number using the number of maximal cliques of the complementary graph. Then, we will also compare the new bound with the existing bounds in the literature. First, we will show that any partition of the set of maximal cliques of $G^c$ can be used to construct a biclique subgraph of $G$.

\begin{lemma}\label{lm:bc_partition}
Given an arbitrary graph $G$ and a set of maximal cliques $\mathcal{K}_{G^c} = \{K^i\}_{i=1}^d$ of its complement $G^c$, we can partition the set $\llbracket d \rrbracket$ into two arbitrary non-empty sets $I$ and $J$ such that $I \cup J = \llbracket d \rrbracket$. If $\bigcup_{i \in I} K^i \setminus \bigcup_{j \in J} K^j \neq \emptyset$ and $\bigcup_{j \in J} K^j \setminus \bigcup_{i \in I} K^i \neq \emptyset$, then $\left\{\bigcup_{i \in I} K^i \setminus \bigcup_{j \in J} K^j, \bigcup_{j \in J} K^j \setminus \bigcup_{i \in I} K^i \right\}$ is a biclique subgraph of $G$.
\end{lemma}

\begin{proof}
Suppose that $L= \bigcup_{i \in I} K^i \setminus \bigcup_{j \in J} K^j$ and $R= \bigcup_{j \in J} K^j \setminus \bigcup_{i \in I} K^i$. 

Given arbitrary $u \in L$ and $v \in R$, $u$ isn't in any maximal clique in $\{K^i\}_{i \in I}$ and $v$ isn't in any maximal clique in $\{K^j\}_{j \in J}$. Hence, $\{u, v\} \not\subseteq K$ for any $K \in \mathcal{K}_{G^c}$. Since the edges of $G^c$ is covered by its maximal cliques, i.e., $E(G^c) = \bigcup_{K \in \mathcal{K}_{G^c}} E(K)$, $\{L, R\}$ is a biclique subgraph of $G$.
\end{proof}

Then, we will show that any biclique subgraph $\{L, R\}$ of a graph $G$, there exists a partition of the set of maximal cliques of $G^c$ such that $\{L, R\}$ is also a biclique subgraph of the biclique constructed by the partition as in Lemma~\ref{lm:bc_partition}.

\begin{proposition} \label{prop:bc_larger_biclique}
Given an arbitrary biclique $\{L, R\}$ of a graph $G$, there exists a partition of the set of all maximal cliques of $G^c$, $\mathcal{K}_{G^c} = \{K^i\}_{i=1}^d$: $\{K^i\}_{i \in I}$ and $\{K^j\}_{j \in J}$ such that $\{L, R\}$ is a subset of another biclique subgraph \\
$\left\{\bigcup_{i \in I} K^i \setminus \bigcup_{j \in J} K^j, \bigcup_{j \in J} K^j \setminus \bigcup_{i \in I} K^i \right\}$ of $G$.
\end{proposition}

\begin{proof}
By definition of the maximal cliques and biclique subgraphs, we know that given arbitrary $u \in L$ and $v \in R$, there does not exist $K \in \mathcal{K}_{G^c}$ such that $\{u, v\} \in K$.

Let's denote that $\mathcal{K}_u = \{K \in \mathcal{K}_{G^c}: u \in K\}$. Then, let $\mathcal{K}_L = \bigcup_{u \in L} \mathcal{K}_u$ and $\mathcal{K}_R = \mathcal{K}_{G^c} \setminus \mathcal{K}_L$. Also, let $I$ and $J$ be the partition of $\llbracket d \rrbracket$ such that $\{K^i\}_{i \in I} = \mathcal{K}_L$ and $\{K^j\}_{j \in J} = \mathcal{K}_R$.

Given an arbitrary $u \in L$, $u \in \bigcup_{K \in \mathcal{K}_u} K \subseteq \bigcup_{K \in \mathcal{K}_L} K$. Since $\mathcal{K}_R = \mathcal{K}_{G^c} \setminus \mathcal{K}_L \subseteq \mathcal{K}_{G^c} \setminus \mathcal{K}_u$ and $u \in \bigcup_{K \in \mathcal{K}_u} K$, then $u \not\in \bigcup_{K \in \mathcal{K}_R} K$. Hence, $L \subseteq \bigcup_{K \in \mathcal{K}_L} K \setminus \bigcup_{K \in \mathcal{K}_R} K$.

Given an arbitrary $v \in R$, $v \not\in \bigcup_{K \in \mathcal{K}_u} K$ for any $u \in L$. Thus, $v \not\in \bigcup_{K \in \mathcal{K}_L} K$. Since $v \in V(G) = \bigcup_{K \in \mathcal{K}_{G^c}} K$, then $v \in \bigcup_{K \in \mathcal{K}_R} K$, which implies that $$R \subseteq \bigcup_{K \in \mathcal{K}_R} K \setminus \bigcup_{K \in \mathcal{K}_L} K.$$

Since both $L$ and $R$ are nonempty, by Lemma~\ref{lm:bc_partition}, $$\left\{\bigcup_{i \in I} K^i \setminus \bigcup_{j \in J} K^j, \bigcup_{j \in J} K^j \setminus \bigcup_{i \in I} K^i \right\}$$ is a biclique subgraph of $G$.
\end{proof}

\begin{remark} \label{rm:bc_nonempty}
Given a nonempty graph $G$ and two different arbitrary maximal cliques of $G^c$, $K^1$ and $K^2$, $G(K^1 \cup K^2)$ has at least one edge.
\end{remark}

\begin{proof}
Assume that $G(K^1 \cup K^2)$ is an empty graph, then $K^1 \cup K^2$ is a clique of $G^c$. Also, since $K^1 \neq K^2$, $K^1$ and $K^2$ are not maximal cliques of $G^c$, which is a contradiction.
\end{proof}

\begin{remark} \label{rm:bc_not_cover}
Given an arbitrary biclique $\{L, R\}$ of a graph $G$, the set of all maximal cliques of $G^c$, $\mathcal{K}_{G^c} = \{K^i\}_{i=1}^d$ and a partition of $\llbracket d \rrbracket$: $I$ and $J$, the biclique $\left\{\bigcup_{i \in I} K^i \setminus \bigcup_{j \in J} K^j, \bigcup_{j \in J} K^j \setminus \bigcup_{i \in I} K^i \right\}$ does not have any edge in $G\left(\bigcup_{i \in I} K^i \right)$ or $G\left(\bigcup_{j \in J} K^j\right)$.
\end{remark}

By using Remarks~\ref{rm:bc_nonempty} and~\ref{rm:bc_not_cover}, we can provide new lower bound of the biclique cover number of a graph $G$: $\bc(G) \geq \lceil \log_2(\mc(G^c)) \rceil$. 

\begin{theorem} \label{thm:bc_log}
Given a graph $G$, $\bc(G) \geq \lceil \log_2(\mc(G^c)) \rceil$.
\end{theorem}

\begin{proof}
Let $\{\{L^k, R^k\}\}_{k=1}^t$ be a minimum biclique cover of $G$. Also, suppose that $\mathcal{K}_{G^c} = \{K^i\}_{i=1}^d$ is the set of all maximal cliques of $G^c$. Then, for each $k \in \llbracket t \rrbracket$, there exists a partition of $\llbracket d \rrbracket$, $I^k$ and $J^k$, such that $\{L^k, R^k\}$ is a subgraph of another biclique subgraph $\left\{\bigcup_{i \in I^k} K^i \setminus \bigcup_{j \in J^k} K^j, \bigcup_{j \in J^k} K^j \setminus \bigcup_{i \in I^k} K^i \right\}$ of $G$ by Proposition~\ref{prop:bc_larger_biclique}. We construct a complete graph $\mathcal{K}_d$ with vertices $\llbracket d \rrbracket$ such that each vertex $i$ represents $K^i$ in the maximal clique sets of $G^c$. By Remark~\ref{rm:bc_nonempty}, we know that $G(K^{i'} \cup K^{j'})$ is nonempty for arbitrary two distinct $i',j' \in \llbracket d \rrbracket$. By Remark~\ref{rm:bc_not_cover}, if the intersection between edges of $G(K^{i'} \cup K^{j'})$ and edges of $\left\{\bigcup_{i \in  I^k} K^i \setminus \bigcup_{j \in J^k} K^j, \bigcup_{j \in J^k} K^j \setminus \bigcup_{i \in I^k} K^i \right\}$ is nonempty, then $i' \in I^k$ and $j' \in J^k$, or $i' \in J^k$ and $j' \in I^k$.

Since $\{\{L^k, R^k\}\}_{k=1}^t$ is a minimum biclique cover of $G$, then arbitrary edge $i'j'$ of $\mathcal{K}_d$ must be in a biclique subgraph $\{I^k, J^k\}$ of $\mathcal{K}_d$ for some $k \in \llbracket t \rrbracket$. Then, we know that $\{\{I^k, J^k\}\}_{k=1}^t$ is a biclique cover of $\mathcal{K}_d$. Since $\bc(\mathcal{K}_d) = \lceil \log_2(d) \rceil$, then $t \geq \lceil \log_2(d) \rceil$, where $t = \bc(G)$ and $d = \mc(G^c)$.
\end{proof}

\begin{theorem}
Given a graph $G = (V, E)$, $\mc(G^c) \geq \chi(G)$.
\end{theorem}

\begin{proof}
Let $\mathcal{K}_{G^c} = \{K^i\}_{i=1}^d$ be the set of all maximal cliques of $G^c$. Then, we construct a color procedure on the vertices of $G$: if a vertex in $K^i$ and not in any $K^j$ for $j \in \llbracket i-1 \rrbracket$, then set color $i$ to it. It is not hard to see that we can color all vertices in $V$ with only $d$ colors. Then, we want to prove that each pair of adjacent vertices in $G$ have different colors. Let $u,v \in V$ be two vertices that share the same color. Then, according to our color procedure, it is not hard to see that $u, v \in K^j$ for some $j \in \llbracket d \rrbracket$. Also, $K^j$ is a clique in $G^c$ so $uv \not\in E$. Hence, $\chi(G) \leq d = \mc(G^c)$.
\end{proof}

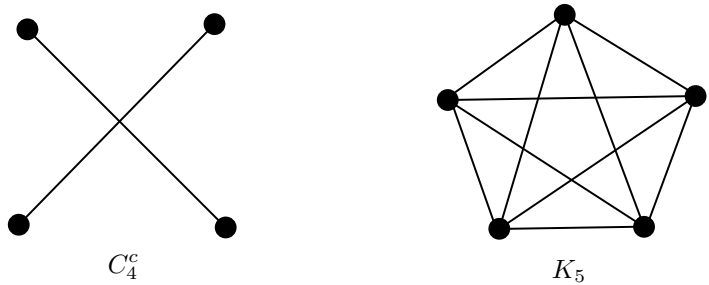
\begin{figure}[H]
    \centering
\tikzset{every picture/.style={line width=0.75pt}} 

\begin{tikzpicture}[x=0.75pt,y=0.75pt,yscale=-1,xscale=1]

\draw    (83,110) -- (183,210) ;
\draw  [fill={rgb, 255:red, 0; green, 0; blue, 0 }  ,fill opacity=1 ] (78,110) .. controls (78,107.24) and (80.24,105) .. (83,105) .. controls (85.76,105) and (88,107.24) .. (88,110) .. controls (88,112.76) and (85.76,115) .. (83,115) .. controls (80.24,115) and (78,112.76) .. (78,110) -- cycle ;
\draw  [fill={rgb, 255:red, 0; green, 0; blue, 0 }  ,fill opacity=1 ] (178,210) .. controls (178,207.24) and (180.24,205) .. (183,205) .. controls (185.76,205) and (188,207.24) .. (188,210) .. controls (188,212.76) and (185.76,215) .. (183,215) .. controls (180.24,215) and (178,212.76) .. (178,210) -- cycle ;
\draw    (177.37,107.38) -- (78.63,208.62) ;
\draw  [fill={rgb, 255:red, 0; green, 0; blue, 0 }  ,fill opacity=1 ] (177.3,102.38) .. controls (180.07,102.34) and (182.33,104.55) .. (182.37,107.31) .. controls (182.4,110.07) and (180.19,112.34) .. (177.43,112.38) .. controls (174.67,112.41) and (172.4,110.2) .. (172.37,107.44) .. controls (172.33,104.68) and (174.54,102.41) .. (177.3,102.38) -- cycle ;
\draw  [fill={rgb, 255:red, 0; green, 0; blue, 0 }  ,fill opacity=1 ] (78.57,203.62) .. controls (81.33,203.59) and (83.6,205.8) .. (83.63,208.56) .. controls (83.67,211.32) and (81.46,213.59) .. (78.7,213.62) .. controls (75.93,213.66) and (73.67,211.45) .. (73.63,208.69) .. controls (73.6,205.93) and (75.81,203.66) .. (78.57,203.62) -- cycle ;
\draw  [fill={rgb, 255:red, 0; green, 0; blue, 0 }  ,fill opacity=1 ] (349,102.67) .. controls (349,99.91) and (351.24,97.67) .. (354,97.67) .. controls (356.76,97.67) and (359,99.91) .. (359,102.67) .. controls (359,105.43) and (356.76,107.67) .. (354,107.67) .. controls (351.24,107.67) and (349,105.43) .. (349,102.67) -- cycle ;
\draw  [fill={rgb, 255:red, 0; green, 0; blue, 0 }  ,fill opacity=1 ] (290,145.67) .. controls (290,142.91) and (292.24,140.67) .. (295,140.67) .. controls (297.76,140.67) and (300,142.91) .. (300,145.67) .. controls (300,148.43) and (297.76,150.67) .. (295,150.67) .. controls (292.24,150.67) and (290,148.43) .. (290,145.67) -- cycle ;
\draw  [fill={rgb, 255:red, 0; green, 0; blue, 0 }  ,fill opacity=1 ] (316,210.67) .. controls (316,207.91) and (318.24,205.67) .. (321,205.67) .. controls (323.76,205.67) and (326,207.91) .. (326,210.67) .. controls (326,213.43) and (323.76,215.67) .. (321,215.67) .. controls (318.24,215.67) and (316,213.43) .. (316,210.67) -- cycle ;
\draw  [fill={rgb, 255:red, 0; green, 0; blue, 0 }  ,fill opacity=1 ] (389,209.67) .. controls (389,206.91) and (391.24,204.67) .. (394,204.67) .. controls (396.76,204.67) and (399,206.91) .. (399,209.67) .. controls (399,212.43) and (396.76,214.67) .. (394,214.67) .. controls (391.24,214.67) and (389,212.43) .. (389,209.67) -- cycle ;
\draw  [fill={rgb, 255:red, 0; green, 0; blue, 0 }  ,fill opacity=1 ] (415,143.67) .. controls (415,140.91) and (417.24,138.67) .. (420,138.67) .. controls (422.76,138.67) and (425,140.91) .. (425,143.67) .. controls (425,146.43) and (422.76,148.67) .. (420,148.67) .. controls (417.24,148.67) and (415,146.43) .. (415,143.67) -- cycle ;
\draw    (350,104.83) -- (298.67,141.83) ;
\draw    (318,206.5) -- (297.67,148.83) ;
\draw    (389,209.67) -- (326,210.67) ;
\draw    (418.33,147.17) -- (396.67,205.83) ;
\draw    (357.33,104.5) -- (416,140.5) ;
\draw    (352.67,104.5) -- (321,210.67) ;
\draw    (394,209.67) -- (295,145.67) ;
\draw    (420,143.67) -- (321,210.67) ;
\draw    (420,143.67) -- (295,145.67) ;
\draw    (354,102.67) -- (394,209.67) ;

\draw (346,225) node [anchor=north west][inner sep=0.75pt]   [align=left] {$\displaystyle K_{5}$};
\draw (122,222) node [anchor=north west][inner sep=0.75pt]   [align=left] {$\displaystyle C_{4}^{c}$};
\end{tikzpicture}
    \caption{Two different examples demonstrate that new lower bound is better than some existing lower bounds.}
    \label{fig:better_bounds}
\end{figure}

We want to note that our new lower bound of biclique cover in Theorem~\ref{thm:bc_log} can be strictly tighter than some existing lower bound in the literature: given a graph $G$, $\bc(G) \geq \lceil \log_2(\chi(G)) \rceil$~\cite{harary1977biparticity}, $\bc(G) \geq \omega(G_E)$~\cite{fishburn1996bipartite}, and $\bc(G) \geq \frac{|M(G)|^2}{|E(G)|}$~\cite{jukna2009covering}, where $\chi(G)$ is chromatic number of $G$, $\omega(G)$ is the clique number of $G$, $M(G)$ is a maximum matching in $G$ and $G_E$ is a graph with the edge set of $G$ as the vertex set and two vertices is adjacent if the corresponding edges have distinct end-nodes and are not included in a cycle of length 4 in $G$. For example, as shown in Figure~\ref{fig:better_bounds}, $\chi(C^c_4) = 2$ and $\mc(C_4) = 4$ where $C^c_4$ is the complementary graph of a cycle with length 4, $C_4$. In another example of Figure~\ref{fig:better_bounds}, $\lceil \log_2(\mc(K^c_5)) \rceil = 3 > \max\left\{\omega(K_{5,E}), \frac{|M(K_5)|^2}{|E(K_5)|} \right\}$. Note that $\omega(K_{5,E}) = 2$ and $\frac{|M(K_5)|^2}{|E(K_5)|} = \frac{2^2}{10} < 1$. Although listing all maximal cliques could take exponential time on general graphs, it could take polynomial time or even linear time on the special class of graphs, like chordal. In the next section, we will focus on a heuristic of biclique cover number of co-chordal graphs.


\section{A Heuristic for Biclique Cover Number of Co-Chordal Graphs} \label{sec:upper_bound}

In this section, we will design a heuristic for finding biclique covers on co-chordal graphs. First, we will show that $\bc(G) \geq \lceil \log_2(\bp(G) + 1) \rceil$ if $G$ is co-chordal. Then, we will show that the biclique cover number of a co-chordal graph $G$ is upper bounded by $\mc(G)-1$. However, there is still a gap between $\lceil \log_2(\mc(G^c)) \rceil$ and $\mc(G^c) - 1$. We want to further tighten the upper bound of the biclique cover number. We want to note that the effort of reducing the gap has been attempted in Algorithm 2 of~\cite{lyu2022modeling} in the setting of combinatorial disjunctive constraints. However, Lyu et al.~\cite{lyu2022modeling} only provided a greedy biclique merging procedure and no theoretical improvements on the upper bound of the biclique cover number than $\mc(G^c) - 1$.  Furthermore, we will introduce optimal edge-rankings on trees and use it to provide a better upper bound of biclique cover number for a co-chordal graph such that each vertex is in at most two independent sets.

\begin{algorithm}[H]
\begin{algorithmic}[1]
\State \textbf{Input}: A clique tree $\mathcal{T}_{\mathcal{K}^c}$ of a chordal graph $G^c$.
\State \textbf{Output}: A biclique partition $\bp$ of the complementary graph $G$ of $G^c$.
\Function{FindPartition}{$\mathcal{T}_{\mathcal{K}^c}$}
\If{$|V(\mathcal{T}_{\mathcal{K}^c})| \leq 1$}
\State \textbf{return} $\emptyset$.
\EndIf
\State Select an arbitrary edge $e$ to cut $\mathcal{T}_{\mathcal{K}^c}$ into two components $\mathcal{T}_{\mathcal{K}^c_1}$ and $\mathcal{T}_{\mathcal{K}^c_2}$. \label{ln:bc_partition_sep_edge}
\State $L \leftarrow \bigcup_{K \in V(\mathcal{T}_{\mathcal{K}^c_1})} V(K) \setminus \MID(e)$; $R \leftarrow \bigcup_{K \in V(\mathcal{T}_{\mathcal{K}^c_2})} V(K) \setminus \MID(e)$
\State \textbf{return} $\{\{L, R\}\} \cup \Call{FindPartition}{\mathcal{T}_{\mathcal{K}^c_1}} \cup \Call{FindPartition}{\mathcal{T}_{\mathcal{K}^c_2}}$
\EndFunction
\end{algorithmic}
\caption{\small Find a biclique partition of a co-chordal graph $G$ given a clique tree of $G^c$.} \label{alg:biclique_partition_heuristic}
\end{algorithm}

We want to show that $\bc(G) \geq \lceil \log_2(\bp(G) + 1) \rceil$ if $G$ is co-chordal. Also, we want to remark that $\bc(G) = \lceil \log_2(\bp(G) + 1) \rceil$ if $G$ is a complete graph which is co-chordal~\footnote{If $G$ is a complete graph with n vertices, $\bc(G) = \lceil \log_2(n) \rceil$ and $\bp(G) = n - 1$.}.

\begin{theorem}
    Given a co-chordal graph $G$, $\bc(G) \geq \lceil \log_2(\bp(G) + 1) \rceil$.
\end{theorem}

\begin{proof}
    Given a co-chordal graph $G$, we know $\bp(G) \leq \mc(G^c) - 1$ by Theorems 1 and 2 of~\cite{lyu2022finding} and $\bc(G) \geq \lceil \log_2(\mc(G^c)) \rceil$ by Theorem~\ref{thm:bc_log}. Thus, $\bc(G) \geq \lceil \log_2(\bp(G) + 1) \rceil$.
\end{proof}

\begin{remark}
    If $G$ is a complete graph, then $\bc(G) = \lceil \log_2(\bp(G) + 1) \rceil$.
\end{remark}

We also want to note that $\bc(G) \geq  \log_2(\bp(G) + 1)$ implies $\bp(G) \leq 2^{\bc(G)} - 1$ if $G$ is co-chordal. It is a better bound than $\bp(G) \leq \frac{1}{2}(3^{\bc(G)} - 1)$~\cite{pinto2013biclique} for any $\bc(G) > 1$ for co-chordal graphs.

\begin{theorem}[Theorems 1 and 2~\cite{lyu2022finding}] \label{thm:bp_upper_bound}
Given a co-chordal graph $G$ and clique tree $\mathcal{T}_{\mathcal{K}^c}$ of its complement $G^c$, the output of $\Call{FindPartition}{\mathcal{T}_{\mathcal{K}^c}}$ is a biclique partition of $G$ with size $\mc(G^c) - 1$.
\end{theorem}

Since a biclique partition is also a biclique cover, we can get Corollary~\ref{cor:bc_upper_cochordal} immediately.

\begin{corollary} \label{cor:bc_upper_cochordal}
Given a co-chordal graph $G$, $\bc(G) \leq \mc(G^c) - 1$.
\end{corollary}

We want to improve the upper bound of $\mc(G^c) - 1$. Here is an example as a motivation. Suppose that we are interested in finding a minimum biclique cover of a graph shown in Figure~\ref{fig:co-path-5}. A biclique cover can be found by Algorithm~\ref{alg:biclique_partition_heuristic}. Note that the edge can be selected arbitrarily in Line~\ref{ln:bc_partition_sep_edge} of Algorithm~\ref{alg:biclique_partition_heuristic}. We select the one that can cut the clique tree most balanced in our example.

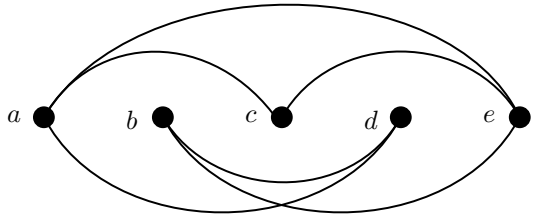
\begin{figure}[H]
    \centering
\tikzset{every picture/.style={line width=0.75pt}} 
\begin{tikzpicture}[x=0.75pt,y=0.75pt,yscale=-1,xscale=1]

\draw  [fill={rgb, 255:red, 0; green, 0; blue, 0 }  ,fill opacity=1 ] (40,165) .. controls (40,162.24) and (42.24,160) .. (45,160) .. controls (47.76,160) and (50,162.24) .. (50,165) .. controls (50,167.76) and (47.76,170) .. (45,170) .. controls (42.24,170) and (40,167.76) .. (40,165) -- cycle ;
\draw  [fill={rgb, 255:red, 0; green, 0; blue, 0 }  ,fill opacity=1 ] (280,165) .. controls (280,162.24) and (282.24,160) .. (285,160) .. controls (287.76,160) and (290,162.24) .. (290,165) .. controls (290,167.76) and (287.76,170) .. (285,170) .. controls (282.24,170) and (280,167.76) .. (280,165) -- cycle ;
\draw  [fill={rgb, 255:red, 0; green, 0; blue, 0 }  ,fill opacity=1 ] (160,165) .. controls (160,162.24) and (162.24,160) .. (165,160) .. controls (167.76,160) and (170,162.24) .. (170,165) .. controls (170,167.76) and (167.76,170) .. (165,170) .. controls (162.24,170) and (160,167.76) .. (160,165) -- cycle ;
\draw  [fill={rgb, 255:red, 0; green, 0; blue, 0 }  ,fill opacity=1 ] (220,165) .. controls (220,162.24) and (222.24,160) .. (225,160) .. controls (227.76,160) and (230,162.24) .. (230,165) .. controls (230,167.76) and (227.76,170) .. (225,170) .. controls (222.24,170) and (220,167.76) .. (220,165) -- cycle ;
\draw  [fill={rgb, 255:red, 0; green, 0; blue, 0 }  ,fill opacity=1 ] (100,165) .. controls (100,162.24) and (102.24,160) .. (105,160) .. controls (107.76,160) and (110,162.24) .. (110,165) .. controls (110,167.76) and (107.76,170) .. (105,170) .. controls (102.24,170) and (100,167.76) .. (100,165) -- cycle ;
\draw    (45,165) .. controls (70,121.25) and (130.5,119.75) .. (163.5,167.25) ;
\draw    (165,165) .. controls (190.5,120.25) and (260,121.25) .. (285,165) ;
\draw    (105,165) .. controls (130.5,207.75) and (200,209.75) .. (225,165) ;
\draw    (45,165) .. controls (91,88.75) and (249,89.75) .. (285,165) ;
\draw    (45,165) .. controls (78.5,228.75) and (190.5,229.25) .. (225,165) ;
\draw    (105,165) .. controls (138.5,228.75) and (250.5,229.25) .. (285,165) ;

\draw (25,160) node [anchor=north west][inner sep=0.75pt]    {$\displaystyle a$};
\draw (85,160) node [anchor=north west][inner sep=0.75pt]    {$\displaystyle b$};
\draw (145,160) node [anchor=north west][inner sep=0.75pt]    {$\displaystyle c$};
\draw (205,160) node [anchor=north west][inner sep=0.75pt]    {$\displaystyle d$};
\draw (265,160) node [anchor=north west][inner sep=0.75pt]    {$\displaystyle e$};

\end{tikzpicture}
    \caption{A co-chordal graph $G$ where $G^c$ is a path graph.}
    \label{fig:co-path-5}
\end{figure}

The procedure to obtain a biclique cover of the graph in Figure~\ref{fig:co-path-5} is demonstrated in Figure~\ref{fig:co-path-5-procedure}. The biclique cover is $\{\{\{a, b\}, \{c, d\}\}, \{\{a\}, \{c\}\}, \{\{c\}, \{e\}\}\}$ with a size 3. However, it is not hard to see that we can merge the biclique $\{\{a\}, \{c\}\}$ with $\{\{c\}, \{e\}\}$ to build a larger biclique $\{\{a, e\}, \{c\}\}$ to obtain a new biclique cover with a size of 2. Note that this biclique cover is the minimum by Theorem~\ref{thm:bc_log}. Also note that the size of the biclique cover is equal to the number of ``levels" minus 1 in this example: we merged all the bicliques in the same ``level" into one biclique and the last level does not contain any biclique.  

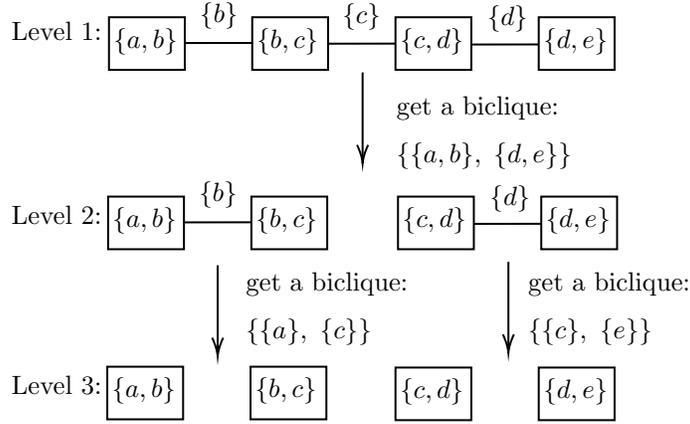
\begin{figure}[H]
    \centering
\tikzset{every picture/.style={line width=0.75pt}} 

\begin{tikzpicture}[x=0.75pt,y=0.75pt,yscale=-0.85,xscale=0.85]
\draw   (59.5,55) -- (104.5,55) -- (104.5,86.25) -- (59.5,86.25) -- cycle ;
\draw    (104.3,70.3) -- (144.5,70.25) ;
\draw   (144.5,55) -- (189.5,55) -- (189.5,86.25) -- (144.5,86.25) -- cycle ;
\draw   (229.5,55) -- (274.5,55) -- (274.5,86.25) -- (229.5,86.25) -- cycle ;
\draw   (314.5,55) -- (359.5,55) -- (359.5,86.25) -- (314.5,86.25) -- cycle ;
\draw    (189.3,70.7) -- (229.5,70.65) ;
\draw    (274.5,71.19) -- (314.7,71.14) ;
\draw    (210,87.75) -- (210,139.75) ;
\draw [shift={(210,141.75)}, rotate = 270] [color={rgb, 255:red, 0; green, 0; blue, 0 }  ][line width=0.75]    (10.93,-3.29) .. controls (6.95,-1.4) and (3.31,-0.3) .. (0,0) .. controls (3.31,0.3) and (6.95,1.4) .. (10.93,3.29)   ;
\draw   (59,161.5) -- (104,161.5) -- (104,192.75) -- (59,192.75) -- cycle ;
\draw    (103.8,176.8) -- (144,176.75) ;
\draw   (144,161.5) -- (189,161.5) -- (189,192.75) -- (144,192.75) -- cycle ;
\draw   (58.5,262.5) -- (103.5,262.5) -- (103.5,293.75) -- (58.5,293.75) -- cycle ;
\draw   (143.5,262.5) -- (188.5,262.5) -- (188.5,293.75) -- (143.5,293.75) -- cycle ;
\draw    (124,202.25) -- (124,254.25) ;
\draw [shift={(124,256.25)}, rotate = 270] [color={rgb, 255:red, 0; green, 0; blue, 0 }  ][line width=0.75]    (10.93,-3.29) .. controls (6.95,-1.4) and (3.31,-0.3) .. (0,0) .. controls (3.31,0.3) and (6.95,1.4) .. (10.93,3.29)   ;
\draw   (231,161.5) -- (276,161.5) -- (276,192.75) -- (231,192.75) -- cycle ;
\draw   (316,161.5) -- (361,161.5) -- (361,192.75) -- (316,192.75) -- cycle ;
\draw    (276,177.69) -- (316.2,177.64) ;
\draw   (229.5,263) -- (274.5,263) -- (274.5,294.25) -- (229.5,294.25) -- cycle ;
\draw   (314.5,263) -- (359.5,263) -- (359.5,294.25) -- (314.5,294.25) -- cycle ;
\draw    (296.5,200.25) -- (296.5,252.25) ;
\draw [shift={(296.5,254.25)}, rotate = 270] [color={rgb, 255:red, 0; green, 0; blue, 0 }  ][line width=0.75]    (10.93,-3.29) .. controls (6.95,-1.4) and (3.31,-0.3) .. (0,0) .. controls (3.31,0.3) and (6.95,1.4) .. (10.93,3.29)   ;

\draw (60.5,59.4) node [anchor=north west][inner sep=0.75pt]    {$\{a,b\} \ $};
\draw (111.78,44.19) node [anchor=north west][inner sep=0.75pt]    {$\{b\} \ $};
\draw (145.5,59.4) node [anchor=north west][inner sep=0.75pt]    {$\{b,c\} \ $};
\draw (230.5,59.4) node [anchor=north west][inner sep=0.75pt]    {$\{c,d\} \ $};
\draw (315.5,59.4) node [anchor=north west][inner sep=0.75pt]    {$\{d,e\} \ $};
\draw (196.28,45.19) node [anchor=north west][inner sep=0.75pt]    {$\{c\} \ $};
\draw (282.28,46.19) node [anchor=north west][inner sep=0.75pt]    {$\{d\} \ $};
\draw (60,165.9) node [anchor=north west][inner sep=0.75pt]    {$\{a,b\} \ $};
\draw (111.28,150.69) node [anchor=north west][inner sep=0.75pt]    {$\{b\} \ $};
\draw (145,165.9) node [anchor=north west][inner sep=0.75pt]    {$\{b,c\} \ $};
\draw (59.5,266.9) node [anchor=north west][inner sep=0.75pt]    {$\{a,b\} \ $};
\draw (144.5,266.9) node [anchor=north west][inner sep=0.75pt]    {$\{b,c\} \ $};
\draw (232,165.9) node [anchor=north west][inner sep=0.75pt]    {$\{c,d\} \ $};
\draw (317,165.9) node [anchor=north west][inner sep=0.75pt]    {$\{d,e\} \ $};
\draw (283.78,152.69) node [anchor=north west][inner sep=0.75pt]    {$\{d\} \ $};
\draw (228.5,100) node [anchor=north west][inner sep=0.75pt]   [align=left] {get a biclique: \\
$\{\{a,b\} ,\ \{d,e\}\}$};
\draw (230.5,267.4) node [anchor=north west][inner sep=0.75pt]    {$\{c,d\} \ $};
\draw (315.5,267.4) node [anchor=north west][inner sep=0.75pt]    {$\{d,e\} \ $};
\draw (139,205) node [anchor=north west][inner sep=0.75pt]   [align=left] {get a biclique: \\
$\{\{a\} ,\ \{c\}\}$};
\draw (306.5, 205) node [anchor=north west][inner sep=0.75pt]   [align=left] {get a biclique: \\
$\{\{c\} ,\ \{e\}\}$};

\draw (0,56) node [anchor=north west][inner sep=0.75pt]   [align=left] {Level 1:};
\draw (0,165) node [anchor=north west][inner sep=0.75pt]   [align=left] {Level 2:};
\draw (0,266) node [anchor=north west][inner sep=0.75pt]   [align=left] {Level 3:};
\end{tikzpicture}
    \caption{A procedure to obtain a biclique cover of the graph in Figure~\protect\ref{fig:co-path-5}.}
    \label{fig:co-path-5-procedure}
\end{figure}


Hence, we want to explore the relationship of ``level" and the size of minimum biclique cover of a co-chordal graph. It is worth to notice that the total ``levels" are different by selecting different edge in Line~\ref{ln:bc_partition_sep_edge} of Algorithm~\ref{alg:biclique_partition_heuristic}. Since we want to create a upper bound of the biclique cover number, we want to find the minimum ``levels". In order to describe the concept of the minimum ``levels" formally, we introduce a notation of edge-rankings for tree graphs: an edge ranking of a tree is a labeling of its edges using positive integers such that given any two edges with the same label, the unique path between those two edges has a edge with a larger label.

\begin{definition}
Given a tree $T = (V, E)$, a mapping $\varphi_T: E \rightarrow \{1, 2, \hdots, r\}$ is a edge-ranking of $T$ if for any $e_1, e_2 \in E$ with $\varphi_T(e_1) = \varphi_T(e_2)$, there exists $e_3 \in E$ on the path between $e_1$ and $e_2$ such that $\varphi_T(e_3) > \varphi_T(e_1) = \varphi_T(e_2)$. The number of the ranks used by $\varphi_T$ is $r$. An optimal edge-ranking number of $T$ is the minimum number of ranks in any edge-ranking of $T$ denoted as $\chi_r'(T)$.
\end{definition}


We first want to remark that the optimal edge-ranking of a tree $T$ has number of ranks no less than the maximum degree of $T$, i.e., $\chi_r'(T) \geq \Delta(T)$. Thus, the minimum number of ranks used in an edge-ranking of a star graph $T$ is $\Delta(T) = V(T) - 1$. Furthermore, for any nonempty tree $T$, $\chi_r'(T) \geq \lceil \log_2(|V(T)|) \rceil$.

\begin{remark}
Given an arbitrary tree $T$, $\chi_r'(T) \geq \Delta(T)$.
\end{remark}

\begin{remark}[Lemma 2.1~\cite{iyer1991edge}]
Given an arbitrary nonempty tree $T$, $\chi_r'(T) \geq \lceil \log_2(|V(T)|) \rceil$.
\end{remark}





\begin{proposition}
Given a path graph $P_n$ with $n$ vertices, $\chi_r'(P_n) = \lceil \log_2(n) \rceil$.
\end{proposition}
\begin{proof}
We prove it by induction. In base step, $\chi_r'(P_1) = 0$ since $P_1$ doesn't have any edge. In induction step, assume that $\chi_r'(P_i) = \lceil \log_2(i) \rceil$ and $\varphi_{P_i}$ be an optimal edge-ranking of $P_i$ for all $i = 1, 2, \hdots, n-1$. Then, we can construct an edge-ranking of $P_n$ by using $\varphi_{P_{\lfloor n/2 \rfloor}}$ and $\varphi_{P_{\lceil n/2 \rceil}}$ and use an edge with label
\begin{align*}
    \lceil \log_2(\lceil n/2 \rceil) \rceil + 1 = \lceil \log_2(n / 2) \rceil + 1 = \lceil \log_2(n) \rceil.
\end{align*}

Thus, $\chi_r'(P_n) = \lceil \log_2(n) \rceil$. Note that $\lceil \log_2(\lceil n/2 \rceil) \rceil = \lceil \log_2(n / 2) \rceil$ for any positive integer $n > 1$.
\end{proof}

We will design a subroutine in Algorithm~\ref{alg:sub_find} such that the edge selection in Algorithm~\ref{alg:biclique_partition_heuristic} is according to an edge-ranking and the bicliques are also stored based on the edge-ranking.



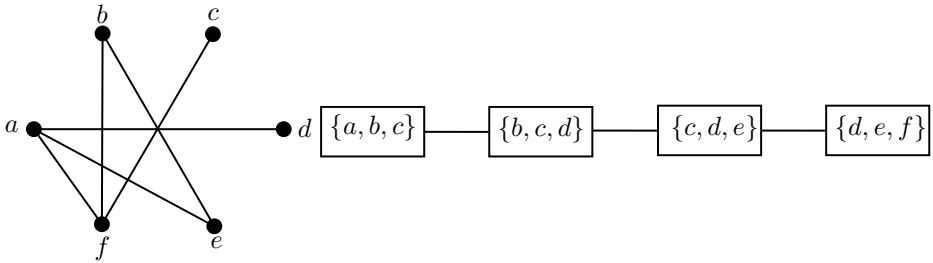
\begin{figure}[H]
    \centering
    \begin{minipage}{0.34\textwidth}
    \centering
    \tikzset{every picture/.style={line width=0.75pt}} 
    \begin{tikzpicture}[x=0.75pt,y=0.75pt,yscale=-0.7,xscale=0.7]
    
    \draw  [fill={rgb, 255:red, 0; green, 0; blue, 0 }  ,fill opacity=1 ] (40,165) .. controls (40,162.24) and (42.24,160) .. (45,160) .. controls (47.76,160) and (50,162.24) .. (50,165) .. controls (50,167.76) and (47.76,170) .. (45,170) .. controls (42.24,170) and (40,167.76) .. (40,165) -- cycle ;
    \draw  [fill={rgb, 255:red, 0; green, 0; blue, 0 }  ,fill opacity=1 ] (170,235) .. controls (170,232.24) and (172.24,230) .. (175,230) .. controls (177.76,230) and (180,232.24) .. (180,235) .. controls (180,237.76) and (177.76,240) .. (175,240) .. controls (172.24,240) and (170,237.76) .. (170,235) -- cycle ;
    \draw  [fill={rgb, 255:red, 0; green, 0; blue, 0 }  ,fill opacity=1 ] (169,96.5) .. controls (169,93.74) and (171.24,91.5) .. (174,91.5) .. controls (176.76,91.5) and (179,93.74) .. (179,96.5) .. controls (179,99.26) and (176.76,101.5) .. (174,101.5) .. controls (171.24,101.5) and (169,99.26) .. (169,96.5) -- cycle ;
    \draw  [fill={rgb, 255:red, 0; green, 0; blue, 0 }  ,fill opacity=1 ] (220,165) .. controls (220,162.24) and (222.24,160) .. (225,160) .. controls (227.76,160) and (230,162.24) .. (230,165) .. controls (230,167.76) and (227.76,170) .. (225,170) .. controls (222.24,170) and (220,167.76) .. (220,165) -- cycle ;
    \draw  [fill={rgb, 255:red, 0; green, 0; blue, 0 }  ,fill opacity=1 ] (89.5,96) .. controls (89.5,93.24) and (91.74,91) .. (94.5,91) .. controls (97.26,91) and (99.5,93.24) .. (99.5,96) .. controls (99.5,98.76) and (97.26,101) .. (94.5,101) .. controls (91.74,101) and (89.5,98.76) .. (89.5,96) -- cycle ;
    \draw  [fill={rgb, 255:red, 0; green, 0; blue, 0 }  ,fill opacity=1 ] (89,233.5) .. controls (89,230.74) and (91.24,228.5) .. (94,228.5) .. controls (96.76,228.5) and (99,230.74) .. (99,233.5) .. controls (99,236.26) and (96.76,238.5) .. (94,238.5) .. controls (91.24,238.5) and (89,236.26) .. (89,233.5) -- cycle ;
    \draw    (45,165) -- (225,165) ;
    \draw    (45,165) -- (175,235) ;
    \draw    (45,165) -- (94,233.5) ;
    \draw    (94.5,96) -- (175,235) ;
    \draw    (94.5,96) -- (94,233.5) ;
    \draw    (174,96.5) -- (94,233.5) ;
    
    \draw (22,157) node [anchor=north west][inner sep=0.75pt]    {$a$};
    \draw (88,73) node [anchor=north west][inner sep=0.75pt]    {$b$};
    \draw (168,76) node [anchor=north west][inner sep=0.75pt]    {$c$};
    \draw (233,155) node [anchor=north west][inner sep=0.75pt]    {$d$};
    \draw (170,240) node [anchor=north west][inner sep=0.75pt]    {$e$};
    \draw (86.5,240) node [anchor=north west][inner sep=0.75pt]    {$f$};

    \end{tikzpicture}
    \end{minipage}
    \begin{minipage}{0.65\textwidth}
    \centering
    \tikzset{every picture/.style={line width=0.75pt}} 
    \begin{tikzpicture}[x=0.75pt,y=0.75pt,yscale=-0.8,xscale=0.8]

\draw   (20.5,147) -- (85.5,147) -- (85.5,178.25) -- (20.5,178.25) -- cycle ;
\draw   (126.5,147) -- (191.5,147) -- (191.5,178.25) -- (126.5,178.25) -- cycle ;
\draw    (85.3,162.7) -- (126,162.75) ;
\draw   (232.9,146.2) -- (297.9,146.2) -- (297.9,177.45) -- (232.9,177.45) -- cycle ;
\draw    (191.7,161.9) -- (232.4,161.95) ;
\draw   (338.9,146.2) -- (403.9,146.2) -- (403.9,177.45) -- (338.9,177.45) -- cycle ;
\draw    (297.7,161.9) -- (338.4,161.95) ;

\draw (24,149.7) node [anchor=north west][inner sep=0.75pt]    {$\{a,b,c\} \ $};
\draw (130.1,150.9) node [anchor=north west][inner sep=0.75pt]    {$\{b,c,d\} \ $};
\draw (239.1,150.1) node [anchor=north west][inner sep=0.75pt]    {$\{c,d,e\} \ $};
\draw (342.5,150.5) node [anchor=north west][inner sep=0.75pt]    {$\{d,e,f\} \ $};

\end{tikzpicture}
    \end{minipage}
    \caption{A co-chordal graph $G$ (left) and a clique tree of $G^c$ (right).}
    \label{fig:co_chordal_counter}
\end{figure}

Our first conjecture is that given an arbitrary co-chordal graph $G$ and a clique tree $\mathcal{T}_{\mathcal{K}^c}$ of $G^c$, $\bc(G) \leq \chi_r'(\mathcal{T}_{\mathcal{K}^c})$. However, it is not correct. As shown in Figure~\ref{fig:co_chordal_counter} (right), the optimal edge-ranking of the clique tree uses 2 ranks since the clique tree is a path with 4 vertices. However, the biclique cover number of the graph $G$ in Figure~\ref{fig:co_chordal_counter} (left) is 3. A biclique cover of $G$ contains 3 bicliques $\{\{a, b\}, \{e, f\}\}, \{\{a\}, \{d\}\}, \{\{c\}, \{f\}\}$. It is also easy to check the size of the biclique cover of $G$ is at least 3 since any pair of edges $ad$, $be$, $cf$ cannot appear in the same biclique. 

\begin{algorithm}[H]
\begin{algorithmic}[1]
\Function{FindBiclique}{$\mathcal{T}, \varphi, \sigma, r, \bc$}
\If{$|V(\mathcal{T})| \leq 1$}
\State \textbf{return}
\EndIf
\State Select the edge $e$ such that $e = \argmax\{\varphi(e'): e' \in \mathcal{T}\}$.
\State Let $\mathcal{T}_1$ and $\mathcal{T}_2$ be the two subtrees of $\mathcal{T} \setminus e$.
\State $\level \leftarrow r + 1 - \varphi(e)$
\State If $\bc[\level]$ is not initialized, $\bc[\level] = ()$.
\State $\ord \leftarrow \min\{\sigma(v): v \in V(\mathcal{T})\}$ \Comment{In order to sort biclique within $\bc[\level]$.}
\State $L \leftarrow \bigcup_{K \in V(\mathcal{T}_1)} V(K) \setminus \MID(e)$; $R \leftarrow \bigcup_{K \in V(\mathcal{T}_2)} V(K) \setminus \MID(e)$
\State Add $(\{L, R\}, \ord)$ into $\bc[\level]$
\State $\Call{FindBiclique}{\mathcal{T}_1, \varphi, \sigma, r, \bc}$
\State $\Call{FindBiclique}{\mathcal{T}_2, \varphi, \sigma, r, \bc}$
\State \textbf{return} 
\EndFunction
\end{algorithmic}
\caption{A subroutine FindBiclique($\mathcal{T}, \varphi, \sigma, r, \bc$)} 
\label{alg:sub_find}
\end{algorithm}

Thus, in general, we cannot say that $\bc(G) \leq \chi_r'(\mathcal{T}_{\mathcal{K}^c})$ given a co-chordal graph $G$ and a clique tree $\mathcal{T}_{\mathcal{K}^c}$ of $G^c$. Thus, $\Call{MergeBiclique}{\bc_{\level}[i], G}$ in Algorithm~\ref{alg:biclique_tot_bc} might return a set with more than one bicliques if $G$ is co-chordal. However, we can show that if $G$ is co-chordal such that each vertex is in at most two maximal independent sets, then $\Call{MergeBiclique}{\bc_{\level}[i], G}$ in Algorithm~\ref{alg:biclique_tot_bc} will always return a set with one biclique, which leads to $\bc(G) \leq \chi_r'(\mathcal{T}_{\mathcal{K}^c})$, where $\mathcal{T}_{\mathcal{K}^c}$ is a clique tree of $G^c$.

\begin{algorithm}[H]
\begin{algorithmic}[1]
\State \textbf{Input}: A co-chordal graph $G$.
\State \textbf{Output}: A biclique cover $\bc$ of the complementary graph $G$ of $G^c$.
\State Initialize a dictionary $\bc_{\level} \leftarrow \{\}$.
\State Compute a clique tree $\mathcal{T}_{\mathcal{K}^c}$ of a chordal graph $G^c$ by Algorithm~\ref{alg:mcs_clique_tree} as shown in~\ref{sec:clique_tree_appendix}.
\State Find an optimal edge-ranking $\varphi$ of $\mathcal{T}_{\mathcal{K}^c}$ by algorithm EDGE-RANKING in~\cite{lam2001optimal} and record the number of ranks of $\varphi$, $r$.
\State Find an ordering $\sigma$ of $\mathcal{T}_{\mathcal{K}^c}$ by traversing all vertices using breadth-first search from an arbitrary leaf node.
\State $\Call{FindBiclique}{\mathcal{T}_{\mathcal{K}^c}, \varphi, \sigma, r, \bc_{\level}}$. \label{ln:FindBiclique}
\State $\bc \leftarrow \{\}$
\For{$i \in \llbracket r \rrbracket$}
\State $\bc \leftarrow \bc \cup \Call{MergeBiclique}{\bc_{\level}[i], G}$.
\EndFor
\end{algorithmic}
\caption{Find a biclique cover of a co-chordal graph $G$ with a size at most $\mc(G^c) - 1$.} \label{alg:biclique_tot_bc}
\end{algorithm}

\begin{lemma} \label{lm:chordal_subtree_subgraph}
    Given a chordal graph $G^c$, a clique tree $\mathcal{T}_{\mathcal{K}^c}$ of $G^c$, and an arbitrary nonempty subtree $\mathcal{T}'$ of $\mathcal{T}_{\mathcal{K}^c}$ with $\mathcal{K}' = V(\mathcal{T}')$, then we can construct a graph $G' = (V', E')$ such that the vertex set $V'= \bigcup_{K \in \mathcal{K}'} K$ and the edge set $E' = \{uv: \exists K \in \mathcal{K}' \text{ such that} \{u, v\} \subseteq K, u \neq v\}$. Then, $G'$ is an induced subgraph of $G^c$. Furthermore, $\mathcal{T}'$ is a clique tree of $G'$.
\end{lemma}

\begin{proof}
    Let $G'' = G^c[V']$ be an induced subgraph of $G^c$ and $E''$ be the edge set of $G''$. Then, it is not hard to see that $E' \subseteq E''$ since $uv$ is an edge in $G^c$ and $u, v \in V'$ for any $uv \in E'$.

    Assume that $E'' \not\subseteq E'$. Then, there must exist $uv \in E''$ such that $\{u, v\} \not\subseteq K$ for any $K \in \mathcal{K}'$. Then, there must exists a maximal clique $K^{uv}$ of $G^c$ such that $\{u, v\} \in K^{u,v}$. Since $u, v \in V'$, then there must exist $K^u, K^v \in \mathcal{K}'$ such that $u \in K^u$ and $v \in K^v$. Since $\mathcal{T}_{\mathcal{K}^c}$ is a clique tree, then the unique path between $K^u$ and $K^v$ must contain $K^{uv}$. Otherwise, the clique-intersection property will not be satisfied. It is a contradiction since $K^{uv}$ is not in $\mathcal{K}'$ so that $\mathcal{T}'$ is not a tree. Hence, $G'$ is an induced subgraph of $G^c$. It is also not hard to see that $\mathcal{K}'$ is a set of all maximal clique of $G'$. Thus, $\mathcal{T}'$ is a clique tree of $G'$.
\end{proof}

\begin{algorithm}[H]
\begin{algorithmic}[1]
\Function{MergeBiclique}{$\bc', G$}
\If{$\bc' = ()$}
\State \textbf{return} $\emptyset$
\EndIf
\State Sort the element $(\{L, R\}, \ord)$ in $\bc'$ by ascending order of $\ord$.
\State $\bc \leftarrow \emptyset$
\For{$(\{L, R\}, \ord)$ in $\bc'$}
\State $\flag \leftarrow \True$
\For{$\{L', R'\} \in \bc$}
\If{$\{L \cup L', R \cup R'\}$ is a biclique cover of $G$}
\State $L' \leftarrow L \cup L'$; $R' \leftarrow R \cup R'$; $\flag \leftarrow \False$
\ElsIf{$\{R \cup L', L \cup R'\}$ is a biclique cover of $G$}
\State $L' \leftarrow R \cup L'$; $R' \leftarrow L \cup R'$; $\flag \leftarrow \False$
\EndIf
\EndFor
\State If $\flag$, then $\bc \leftarrow \bc \cup \{\{L, R\}\}$.
\EndFor
\State \textbf{return} $\bc$
\EndFunction
\end{algorithmic}
\caption{A subroutine MergeBiclique($\bc', G$)} 
\label{alg:sub_bc}
\end{algorithm}

In order to simplify the notation, we use $G[\mathcal{T}]$ to represent the induced subgraph $G[\bigcup_{K \in V(\mathcal{T})} K]$ where $\mathcal{T}$ is a subtree of a clique tree of $G$ or $G^c$. Also, if $\mathcal{T}$ is a clique tree and $\mathcal{T}_1, \mathcal{T}_2$ are two disjointed subtrees of $\mathcal{T}$ (subtrees do not share any vertex), we denote $\mathcal{T}_1 \oplus \mathcal{T}_2$ as the subtree of $\mathcal{T}$ with the minimum size that contains both $\mathcal{T}_1$ and $\mathcal{T}_2$. In other word, $\mathcal{T}_1 \oplus \mathcal{T}_2$ contains $\mathcal{T}_1$, $\mathcal{T}_2$, and the path between $\mathcal{T}_1$ and $\mathcal{T}_2$.

\begin{proposition} \label{prop:co_chordal_merge_larger_biclique}
Suppose that we have a co-chordal graph $G$ such that each vertex is in at most two maximal independent sets, a set of all maximal cliques $\mathcal{K}^c$ of $G^c$, and a clique tree $\mathcal{T}_{\mathcal{K}^c}$ of $G^c$. Given two disjointed subtree of $\mathcal{T}_{\mathcal{K}^c}$, $\mathcal{T}_1$ and $\mathcal{T}_2$, then an arbitrary biclique subgraph of $G[\mathcal{T}_1]$, $\{L^1, R^1\}$, and an arbitrary biclique subgraph of $G[\mathcal{T}_2]$, $\{L^2, R^2\}$, can be merged to a larger biclique of $G[\mathcal{T}_1 \oplus \mathcal{T}_2]$, i.e. $\{L^1 \cup L^2, R^1 \cup R^2\}$ or $\{L^1 \cup R^2, R^1 \cup L^2\}$ is a biclique subgraph of $G[\mathcal{T}_1 \oplus \mathcal{T}_2]$.
\end{proposition}

\begin{proof}
By Lemma~\ref{lm:chordal_subtree_subgraph}, we know that $\mathcal{T}_1$ is a clique tree of $G^c[\mathcal{T}_1]$ and $\mathcal{T}_2$ is a clique tree of $G^c[\mathcal{T}_2]$. Let $\mathcal{K}_1$ and $\mathcal{K}_2$ be the vertex sets of $\mathcal{T}_1$ and $\mathcal{T}_2$ respectively. By the definition of clique trees, $\mathcal{K}_1$ and $\mathcal{K}_2$ are also the set of all maximal cliques of $G^c[\mathcal{T}_1]$ and $G^c[\mathcal{T}_2]$. By Proposition~\ref{prop:bc_larger_biclique}, there exists a partition of $\mathcal{K}_1$, $\mathcal{K}^L_1$ and $\mathcal{K}^R_1$, such that $\{L^1, R^1\}$ is a subgraph of biclique $\left\{\bigcup_{K \in \mathcal{K}^L_1} K \setminus \bigcup_{K \in \mathcal{K}^R_1} K, \bigcup_{K \in \mathcal{K}^R_1} K \setminus \bigcup_{K \in \mathcal{K}^L_1} K\right\}$. Similarly, $\mathcal{K}^L_2$ and $\mathcal{K}^R_2$ are a partition of $\mathcal{K}_2$ such that $\{L^2, R^2\}$ is a subgraph of biclique $$\left\{\bigcup_{K \in \mathcal{K}^L_2} K \setminus \bigcup_{K \in \mathcal{K}^R_2} K, \bigcup_{K \in \mathcal{K}^R_2} K \setminus \bigcup_{K \in \mathcal{K}^L_2} K \right\}.$$ 

Let $\mathcal{P}$ be the path between $\mathcal{T}_1$ and $\mathcal{T}_2$ in $\mathcal{T}_{\mathcal{K}^c}$ and $\mathcal{K}_{\mathcal{P}}$ be the vertex set of $\mathcal{P}$. Note that $\mathcal{K}_{\mathcal{P}}$ might be empty. Then, without loss of generality, we assume that $\mathcal{K}^L_1$ contains the vertex in $\mathcal{T}_1$ that is adjacent to a vertex in $\mathcal{P} \oplus \mathcal{T}_2$ in the clique tree $\mathcal{T}_{\mathcal{K}^c}$. Similarly, we assume that $\mathcal{K}^L_2$ contains the vertex in $\mathcal{T}_2$ that is adjacent to a vertex in $\mathcal{P} \oplus \mathcal{T}_1$. Then, let $A = \bigcup_{K \in \mathcal{K}^L_1 \cup \mathcal{K}^L_2 \cup \mathcal{K}_{\mathcal{P}}} K$ and $B = \bigcup_{K \in \mathcal{K}^R_1 \cup \mathcal{K}^R_2} K$. By Lemma~\ref{lm:bc_partition}, we know that $\{A \setminus B, B \setminus A\}$ is a biclique subgraph of $G[\mathcal{T}_1 \oplus \mathcal{T}_2]$.

Since each vertex is in at most two maximal independent sets and $\mathcal{T}_{\mathcal{K}^c}$ satisfy the clique-intersection property of clique trees, then if both $u \in K$ and $u \in K'$ for $K, K' \in \mathcal{K}^c$, then $K$ and $K'$ must be adjacent in $\mathcal{T}_{\mathcal{K}^c}$. Since $\mathcal{K}^L_1$ is not empty, then 
an arbitrary vertex in $\mathcal{K}^R_1$ is not adjacent to any vertex in $\mathcal{K}^L_2 \cup \mathcal{K}_{\mathcal{P}}$. Thus, $\bigcup_{K \in \mathcal{K}^R_1} K \cap \bigcup_{K \in \mathcal{K}^L_2 \cup \mathcal{K}_{\mathcal{P}}} K = \emptyset$. Similarly, an arbitrary vertex in $\mathcal{K}^R_2$ is not adjacent to any vertex in $\mathcal{K}^L_1 \cup \mathcal{K}_{\mathcal{P}}$. Thus, $\bigcup_{K \in \mathcal{K}^R_2} K \cap \bigcup_{K \in \mathcal{K}^L_1 \cup \mathcal{K}_{\mathcal{P}}} K = \emptyset$. Then,
\begin{align*}
    &\bigcup_{K \in \mathcal{K}^L_1} K \setminus \bigcup_{K \in \mathcal{K}^R_1} K = \bigcup_{K \in \mathcal{K}^L_1} K \setminus \bigcup_{K \in \mathcal{K}^R_1 \cup \mathcal{K}^R_2} K \subseteq A \setminus B, \\
    &\bigcup_{K \in \mathcal{K}^R_1} K \setminus \bigcup_{K \in \mathcal{K}^L_1} K = \bigcup_{K \in \mathcal{K}^R_1} K \setminus \bigcup_{K \in \mathcal{K}^L_1 \cup \mathcal{K}^L_2 \cup \mathcal{K}^{\mathcal{P}}} K \subseteq B \setminus A.
\end{align*}

Hence, both $\{L^1, R^1\}$ and $\left\{\bigcup_{K \in \mathcal{K}^L_1} K \setminus \bigcup_{K \in \mathcal{K}^R_1} K, \bigcup_{K \in \mathcal{K}^R_1} K \setminus \bigcup_{K \in \mathcal{K}^L_1} K\right\}$ are biclique subgraphs of $\{A \setminus B, B \setminus A\}$. Similarly, $\{L^2, R^2\}$ is a biclique subgraph of $\{A \setminus B, B \setminus A\}$. Therefore, $\{L^1 \cup L^2, R^1 \cup R^2\}$ is a biclique subgraph of $G[\mathcal{T}_1 \oplus \mathcal{T}_2]$. 
\end{proof}

\begin{lemma} \label{lm:disjoint_seq_tree}
    Given a tree $T$ and an ordering $\sigma$ of $T$ by traversing all vertices using breadth-first search from an arbitrary leaf node, we have a sequence of mutually disjoined subtrees $\{T_i\}_{i=1}^n$ such that $\min\{\sigma(v): v \in V(T_i)\} < \min\{\sigma(v): v \in V(T_j)\}$ for integers $1 \leq i < j \leq n$. Then, $T_1 \oplus \hdots \oplus T_{i-1}$ and $T_j$ are also disjointed for any integer $2 \leq i \leq j \leq n$.
\end{lemma}

\begin{proof}
    It is sufficient to prove that $T_1 \oplus T_2$ and $T_j$ are disjointed for $3 \leq j \leq n$. Then, a new sequence of mutually disjointed subtrees $\{T'_i\}_{i=1}^{n-1}$ can be constructed such that $T'_1 = T_1 \oplus T_2$ and $T'_{j-1} = T_j$ for $3 \leq j \leq n$. Since $\min\{\sigma(v): v \in V(T_1 \oplus T_2)\} < \min\{\sigma(v): v \in V(T_1)\}$, it is also not hard to see that $\min\{\sigma(v): v \in V(T'_i)\} \leq \min\{\sigma(v): v \in V(T'_j)\}$ for integers $1 \leq i < j \leq n-1$.

    Let $P$ be the path between $T_1$ and $T_2$. Since $\sigma$ is obtained by a breadth-first search, then $\max\{\sigma(v): v \in V(P)\} < \min\{\sigma(v): v \in V(T_2)\}$. Thus, $V(P) \cap V(T_j) = \emptyset$ and $V(T_1 \oplus T_2) \cap V(T_j) = \emptyset$ for $3 \leq j \leq n$.
\end{proof}


\begin{theorem} \label{thm:co_chordal_edge_ranking_bound}
Given a co-chordal graph $G$ such that each vertex is in at most two maximal independent sets and a clique tree $\mathcal{T}_{\mathcal{K}^c}$ of $G^c$, Algorithm~\ref{alg:biclique_tot_bc} will return a biclique cover of $G$ with a size at most $\chi_r'(\mathcal{T}_{\mathcal{K}^c})$ so that $\bc(G) \leq \chi_r'(\mathcal{T}_{\mathcal{K}^c})$.
\end{theorem}

\begin{proof}
    By Theorem~\ref{thm:bp_upper_bound}, the collection of all bicliques in $\bc_{\level}$ after calling FindBiclique in  Line~\ref{ln:FindBiclique} of Algorithm~\ref{alg:biclique_tot_bc} is a biclique partition of $G$ so that it is a biclique cover of $G$. 
    
    The number of ``level" in $\bc_{\level}$ is equal to $\chi_r'(\mathcal{T}_{\mathcal{K}^c})$. According to Lemma~\ref{lm:disjoint_seq_tree} and Proposition~\ref{prop:co_chordal_merge_larger_biclique}, all the biclique in $\bc_{\level}[i]$ can be merged into a biclique subgraph of $G$ by calling MergeBiclique($\bc_{\level}[i], G$).
\end{proof}

\begin{corollary}
    If $G$ is co-chordal such that each vertex is in at most two maximal independent sets and $G^c$ has a clique tree with an optimal edge-ranking number of $\lceil \log_2(\mc(G^c)) \rceil$, then $\bc(G) = \lceil \log_2(\mc(G^c)) \rceil$.
\end{corollary}

Since a path graph is chordal and its clique tree is also a path, Theorem~\ref{thm:co_chordal_edge_ranking_bound} and Theorem~\ref{thm:bc_log} can provide the biclique cover number of a co-path graph. Note that co-path graphs are conflict graphs associated with special ordered sets type 2 and the biclique covers of those conflict graphs can be used to construct mixed-integer programming formulations~\cite{huchette2019combinatorial}.

\begin{corollary}
Given a co-path graph $P_n^c$ with $n$ vertices, $\bc(P^c_n) = \lceil \log_2(n-1) \rceil$.
\end{corollary}

We also want to remark that Algorithm~\ref{alg:biclique_tot_bc} can return a biclique cover with a size of $\lceil \log_2(\mc(G^c)) \rceil$ of $G$ if $G$ is a co-windmill graph: a clique tree of a windmill graph can be a path, all the maximal cliques in $G^c$ share the same vertex, and all the other vertices in $G^c$ are in only one maximal clique.

\begin{remark}
Given a co-windmill graph $G$, $\bc(G) = \lceil \log_2(\mc(G^c)) \rceil$.
\end{remark}

\section{Conclusions and Future Work} \label{sec:c_fw_bc}

Given an arbitrary graph $G$, $\bc(G) \geq \lceil \log_2(\mc(G^c)) \rceil$. We also showed that $\mc(G^c) \geq \chi(G)$. Thus, $$\bc(G) \geq \lceil \log_2(\mc(G^c)) \rceil \geq \lceil \log_2(\chi(G)) \rceil.$$ We discussed the cases where $\lceil \log_2(\mc(G^c)) \rceil$ provides a strictly better lower bound than $\log_2(\chi(G)) \rceil$, $\omega(G_E)$ and $\frac{|M(G)|^2}{|E(G)|}$ for the biclique cover number $\bc(G)$. If $G$ is a co-chordal graph, $\bc(G) \geq \lceil \log_2(\bp(G) + 1) \rceil$ or $\bp(G) \leq 2^{\bc(G)} - 1$. It is a better bound than $\bp(G) \leq \frac{1}{2}(3^{\bc(G)} - 1)$~\cite{pinto2013biclique} if $G$ is co-chordal and $\bc(G) > 1$. Then, motivated by the heuristic for biclique partitions in Algorithm~\ref{alg:biclique_partition_heuristic} and Algorithm 2 in~\cite{lyu2022modeling}, we designed a polynomial-time heuristic for biclique covers on co-chordal graphs in Algorithm~\ref{alg:biclique_tot_bc}. We then showed that if the graph $G$ is co-chordal and each vertex of $G$ is in at most two independent sets, then $\bc(G) \leq \chi_r'(\mathcal{T}_{\mathcal{K}^c})$ for an arbitrary clique tree $\mathcal{T}_{\mathcal{K}^c}$ of $G^c$. Furthermore, we showed that our heuristic in Algorithm~\ref{alg:biclique_tot_bc} can return a minimum biclique cover if we make a further assumption that $G^c$ has a clique tree with an optimal edge-ranking number of $\lceil \log_2(\mc(G^c)) \rceil$.

We also want to remark some future directions. We have shown that $\bc(G) \leq \chi_r'(\mathcal{T}_{\mathcal{K}^c})$ for an arbitrary clique tree $\mathcal{T}_{\mathcal{K}^c}$ of $G^c$ if the graph $G$ is co-chordal and each vertex of $G$ is in at most two independent sets. Could we identify other subclass of co-chordal with such upper bound of biclique cover number? Could we design a polynomial-time algorithm to find the clique tree of a chordal graph with the minimum optimal edge-ranking number?

\section*{Acknowledgements}

\bibliography{mybibfile}

\begin{thebibliography}{10}
\expandafter\ifx\csname url\endcsname\relax
  \def\url#1{\texttt{#1}}\fi
\expandafter\ifx\csname urlprefix\endcsname\relax\def\urlprefix{URL }\fi
\expandafter\ifx\csname href\endcsname\relax
  \def\href#1#2{#2} \def\path#1{#1}\fi

\bibitem{amilhastre1999fa}
J.~Amilhastre, P.~Janssen, M.-C. Vilarem, Fa minimisation heuristics for a
  class of finite languages, in: International Workshop on Implementing
  Automata, Springer, 1999, pp. 1--12.

\bibitem{amilhastre1998complexity}
J.~Amilhastre, M.-C. Vilarem, P.~Janssen, Complexity of minimum biclique cover
  and minimum biclique decomposition for bipartite domino-free graphs, Discrete
  applied mathematics 86~(2-3) (1998) 125--144.

\bibitem{blair1993introduction}
J.~R. Blair, B.~Peyton, An introduction to chordal graphs and clique trees, in:
  Graph theory and sparse matrix computation, Springer, 1993, pp. 1--29.

\bibitem{chalermsook2014nearly}
P.~Chalermsook, S.~Heydrich, E.~Holm, A.~Karrenbauer, Nearly tight
  approximability results for minimum biclique cover and partition, in:
  European Symposium on Algorithms, Springer, 2014, pp. 235--246.

\bibitem{de1981boolean}
D.~de~Caen, D.~A. Gregory, N.~J. Pullman, The boolean rank of zero-one
  matrices, in: Proceedings of the Third Caribbean Conference on Combinatorics
  and Computing, Barbados, 1981, pp. 169--173.

\bibitem{de1995optimal}
P.~de~la Torre, R.~Greenlaw, A.~A. Sch{\"a}ffer, Optimal edge ranking of trees
  in polynomial time, Algorithmica 13~(6) (1995) 592--618.

\bibitem{ene2008fast}
A.~Ene, W.~Horne, N.~Milosavljevic, P.~Rao, R.~Schreiber, R.~E. Tarjan, Fast
  exact and heuristic methods for role minimization problems, in: Proceedings
  of the 13th ACM symposium on Access control models and technologies, 2008,
  pp. 1--10.

\bibitem{epasto2018efficient}
A.~Epasto, E.~Upfal, Efficient approximation for restricted biclique cover
  problems, Algorithms 11~(6) (2018) 84.

\bibitem{fishburn1996bipartite}
P.~C. Fishburn, P.~L. Hammer, Bipartite dimensions and bipartite degrees of
  graphs, Discrete Mathematics 160~(1-3) (1996) 127--148.

\bibitem{fulkerson1965incidence}
D.~Fulkerson, O.~Gross, Incidence matrices and interval graphs, Pacific journal
  of mathematics 15~(3) (1965) 835--855.

\bibitem{gruber2007inapproximability}
H.~Gruber, M.~Holzer, Inapproximability of nondeterministic state and
  transition complexity assuming p$\ne$ np, in: International Conference on
  Developments in Language Theory, Springer, 2007, pp. 205--216.

\bibitem{gunluk2007new}
O.~G{\"u}nl{\"u}k, A new min-cut max-flow ratio for multicommodity flows, SIAM
  Journal on Discrete Mathematics 21~(1) (2007) 1--15.

\bibitem{guo2018biclique}
K.~Guo, T.~Huynh, M.~Macchia, The biclique covering number of grids, The
  Electronic Journal of Combinatorics 26.
\newblock \href {http://dx.doi.org/10.37236/8316} {\path{doi:10.37236/8316}}.

\bibitem{habib2000lex}
M.~Habib, R.~McConnell, C.~Paul, L.~Viennot, Lex-bfs and partition refinement,
  with applications to transitive orientation, interval graph recognition and
  consecutive ones testing, Theoretical Computer Science 234~(1-2) (2000)
  59--84.

\bibitem{harary1977biparticity}
F.~Harary, D.~Hsu, Z.~Miller, The biparticity of a graph, Journal of graph
  theory 1~(2) (1977) 131--133.

\bibitem{hirsch2006biclique}
M.~Hirsch, H.~Meijer, D.~Rappaport, Biclique edge cover graphs and confluent
  drawings, in: International Symposium on Graph Drawing, Springer, 2006, pp.
  405--416.

\bibitem{huchette2019combinatorial}
J.~Huchette, J.~P. Vielma, A combinatorial approach for small and strong
  formulations of disjunctive constraints, Mathematics of Operations Research
  44~(3) (2019) 793--820.

\bibitem{iyer1991edge}
A.~V. Iyer, H.~D. Ratliff, G.~Vijayan, On an edge ranking problem of trees and
  graphs, Discrete Applied Mathematics 30~(1) (1991) 43--52.

\bibitem{jukna2009covering}
S.~Jukna, A.~S. Kulikov, On covering graphs by complete bipartite subgraphs,
  Discrete Mathematics 309~(10) (2009) 3399--3403.

\bibitem{lam2001optimal}
T.~W. Lam, F.~L. Yue, Optimal edge ranking of trees in linear time,
  Algorithmica 30~(1) (2001) 12--33.

\bibitem{lyu2022finding}
B.~Lyu, I.~V. Hicks, Finding biclique partitions of co-chordal graphs, arXiv
  preprint arXiv:2203.02837\href {http://dx.doi.org/10.48550/ARXIV.2203.02837}
  {\path{doi:10.48550/ARXIV.2203.02837}}.

\bibitem{lyu2022modeling}
B.~Lyu, I.~V. Hicks, J.~Huchette, Modeling combinatorial disjunctive
  constraints via junction trees, arXiv preprint arXiv:2205.06916\href
  {http://dx.doi.org/10.48550/ARXIV.2205.06916}
  {\path{doi:10.48550/ARXIV.2205.06916}}.

\bibitem{monson1995survey}
S.~D. Monson, N.~J. Pullman, R.~Rees, A survey of clique and biclique coverings
  and factorizations of (0, 1)-matrices, Bull. Inst. Combin. Appl 14 (1995)
  17--86.

\bibitem{muller1996edge}
H.~M{\"u}ller, On edge perfectness and classes of bipartite graphs, Discrete
  Mathematics 149~(1-3) (1996) 159--187.

\bibitem{nau1978mathematical}
D.~S. Nau, G.~Markowsky, M.~A. Woodbury, D.~B. Amos, A mathematical analysis of
  human leukocyte antigen serology, Mathematical Biosciences 40~(3-4) (1978)
  243--270.

\bibitem{nor2012mod}
I.~Nor, D.~Hermelin, S.~Charlat, J.~Engelstadter, M.~Reuter, O.~Duron, M.-F.
  Sagot, Mod/resc parsimony inference: Theory and application, Information and
  Computation 213 (2012) 23--32.

\bibitem{orlin1977contentment}
J.~Orlin, Contentment in graph theory: covering graphs with cliques, in:
  Indagationes Mathematicae (Proceedings), Vol.~80, Elsevier, 1977, pp.
  406--424.

\bibitem{pinto2013biclique}
T.~Pinto, Biclique covers and partitions, The Electronic Journal of
  Combinatorics 21 (2014) \!\!.

\bibitem{rose1976algorithmic}
D.~J. Rose, R.~E. Tarjan, G.~S. Lueker, Algorithmic aspects of vertex
  elimination on graphs, SIAM Journal on computing 5~(2) (1976) 266--283.

\bibitem{simon1990approximate}
H.~U. Simon, On approximate solutions for combinatorial optimization problems,
  SIAM Journal on Discrete Mathematics 3~(2) (1990) 294--310.

\bibitem{tarjan1984simple}
R.~E. Tarjan, M.~Yannakakis, Simple linear-time algorithms to test chordality
  of graphs, test acyclicity of hypergraphs, and selectively reduce acyclic
  hypergraphs, SIAM Journal on computing 13~(3) (1984) 566--579.

\bibitem{zhou1994efficient}
X.~Zhou, T.~Nishizeki, An efficient algorithm for edge-ranking trees, in:
  Algorithms—ESA'94: Second Annual European Symposium Utrecht, The
  Netherlands, September 26--28, 1994 Proceedings 2, Springer, 1994, pp.
  118--129.

\bibitem{zhou1995finding}
X.~Zhou, T.~Nishizeki, Finding optimal edge-rankings of trees, in: SODA, 1995,
  pp. 122--131.

\end{thebibliography}

\newpage
\appendix
\section{An Algorithm for Clique Trees of Chordal Graphs}~\label{sec:clique_tree_appendix}
Blair and Peyton~\cite{blair1993introduction} expanded the MCS algorithm to compute a clique tree of a chordal graph $G$ in polynomial time shown in Algorithm~\ref{alg:mcs_clique_tree}.

\begin{algorithm}[H]
\begin{algorithmic}[1]
\State \textbf{Input}: A chordal graph $G$.
\State \textbf{Output}: A clique tree $\mathcal{T}_{\mathcal{K}} =(\mathcal{K}, \mathcal{E}_{\mathcal{T}})$ and a set of all maximal clique $\mathcal{K}$ of $G$.
\State $\prevc \leftarrow 0$. $\mathcal{L}_{n+1} \leftarrow \emptyset$. $s \leftarrow 0$. $\mathcal{E}_{\mathcal{T}} \leftarrow \emptyset$. $\mathcal{K} \leftarrow \emptyset$.
\For{$i = n, n-1, \hdots, 1$}
\State Choose a vertex $v \in V - \mathcal{L}_{i+1}$ for which $|N_G(v) \cap \mathcal{L}_{i+1}|$ is maximum.
\State $\alpha(v) \leftarrow i$
\State $\newc \leftarrow |N_G(v) \cap \mathcal{L}_{i+1}|$
\If{$\newc \leq \prevc$}
\State $\mathcal{K} \leftarrow \mathcal{K} \cup \{K^s\}$
\State $s \leftarrow s + 1$
\State $K^s \leftarrow N_G(v) \cap \mathcal{L}_{i+1}$
\If{$\newc \neq 0$}
\State $u \leftarrow \argmin_{u' \in K^s}\{j | \alpha(u') = j\}$.
\State $p \leftarrow \clique(u)$
\State $\mathcal{E}_{\mathcal{T}} \leftarrow \mathcal{E}_{\mathcal{T}} \cup \{K^sK^p\}$
\EndIf
\EndIf
\State $\clique(v) \leftarrow s$
\State $K^s \leftarrow K^s \cup \{v\}$
\State $\mathcal{L} \leftarrow \mathcal{L}_{i+1} \cup \{v\}$
\State $\prevc \leftarrow \newc$
\EndFor
\State \textbf{return} $(\mathcal{K}, \mathcal{E}_{\mathcal{T}})$, $\mathcal{K}$
\end{algorithmic}
\caption{An algorithm that compute a clique tree of a chordal graph $G$~\cite{blair1993introduction}.} \label{alg:mcs_clique_tree}
\end{algorithm}

\end{document}